\documentclass[10pt]{amsart}
\usepackage{amsfonts}
\usepackage{ifthen}
\usepackage{amsthm}
\usepackage{amsmath,mathrsfs}
\usepackage{graphicx}
\usepackage{amscd,amssymb,amsthm}
\usepackage{color}
\usepackage{hyperref}

\newcounter{minutes}
\divide\time by 60
\newcounter{hours}
\multiply\time by 60 \addtocounter{minutes}{-\time}

\newtheorem{theorem}{Theorem}
\newtheorem{lemma}{Lemma}
\newtheorem{propo}{Proposition}

\newtheorem{corollary}{Corollary}
\newtheorem{remark}{Remark}

\title[New Hermite-Hadamard type integral inequalities  for convex functions and applications]{New Hermite-Hadamard type integral inequalities  for convex functions and theirs applications}

\author[K. Mehrez, P. Agarwal]{Khaled Mehrez and  Praveen Agarwal}
\address{Khaled Mehrez. D\'epartement de Math\'ematiques ISSAT Kasserine,
Universit\'e de Kairouan, Tunisia.}
\email{k.mehrez@yahoo.fr}
\address{Praveen Agarwal. Department of mathematics, Anand International college of engineering,Jaipur,Rajasthan, India}
\email{goyal.praveen@gmail.com}


\begin{document}

\def\thefootnote{}
\footnotetext{ \texttt{File:~\jobname .tex,
          printed: \number\year-0\number\month-\number\day,
          \thehours.\ifnum\theminutes<10{0}\fi\theminutes}
} \makeatletter\def\thefootnote{\@arabic\c@footnote}\makeatother

\maketitle

\begin{abstract}
In this paper, we establish (presumably new type) integral inequalities for convex functions
via the Hermite--Hadamard's inequalities. As applications, we apply these new inequalities to construct inequalities involving special means of real numbers, some error estimates for the formula midpoint are given. Finally, new inequalities for some special and $q-$special functions  are also pointed out.
\end{abstract}

\section{\bf Introduction}
\setcounter{equation}{0}
he theory of convexity is its
 close relationship with theory of inequalities. Many inequalities known in the literature are direct consequences of the applications of convex functions. An important
inequality for convex functions which has been extensively studied in recent decades is Hermite--Hadamard's inequality, which
was obtained by Hermite and Hadamard independently.
To be more precise, a function $f:I\subseteq\mathbb{R}\longrightarrow\mathbb{R}$
is a convex function, $a , b\in I$ with $a < b,$ if and only if,
\begin{equation}\label{1}
f\left(\frac{a+b}{2}\right)\leq\frac{1}{b-a}\int_a^b f(x)dx\leq\frac{f(a)+f(b)}{2},
\end{equation}
which is known as Hermite-?adamard inequality. This inequalities (\ref{1}) has become an important cornerstone in probability and optimization. An account on the history of this inequality can be found in \cite{M}. The aim of this paper is to establish some new  results connected with the Hermite--Hadamard inequalities (\ref{1}). As application we derive   some  elementary  inequalities for real numbers, some error estimates for the formula midpoint and we established new type inequalities for the modified Bessel functions of the first and second kind and the $q-$digamma function.
\section{Some preliminary Lemmas}

In this section, we state the following Lemmas, which are useful in the proofs of our main results.

\begin{lemma}\cite{D}\label{l3}Let $f: I\subseteq\mathbb{R}\longrightarrow\mathbb{R}$ be  a  differentiable
 mapping  on $I^0,$ and $a,b\in I$ with $a<b$, then we have
\begin{equation}
\frac{f(b)+f(a)}{2}-\frac{1}{b-a}\int_a^bf(x)dx=\frac{(b-a)^2}{2}
\int_0^1 t(1-t)f^{\prime\prime}(ta+(1-t)b)dt.
\end{equation}
\end{lemma}
\begin{lemma}\cite{U}\label{l2}Let $f: I\subseteq\mathbb{R}\longrightarrow\mathbb{R}$ be  a  differentiable
 mapping  on $I^0,$ and $a,b\in I$ with $a<b$, then we have
\begin{equation}
\frac{1}{b-a}\int_a^b f(x)dx-f\left(\frac{a+b}{2}\right)=(b-a)\left[\int_0^{1/2}t
 f^\prime(b+(a-b)t)dt+\int_{1/2}^1(t-1) f^\prime(b+(a-b)t)dt\right].
\end{equation}
\end{lemma}
\begin{lemma}\label{TT1}Let $f:I\subseteq\mathbb{R}\longrightarrow\mathbb{R}$
be a  differentiable  mapping  on $I^{0}$ with $a<b.$ If $f$ is a convex function,  then the following inequalities holds:
\begin{equation}\label{K1}
f\left(\frac{a+b}{2}\right)\leq \frac{1}{b-a}\int_a^b f(x)dx
\leq\frac{2f\left(\frac{a+b}{2}\right)+f\left(\frac{3b-a}{2}\right)+
f\left(\frac{3a-b}{2}\right)}{4},
\end{equation}
and
\begin{equation}\label{K2}
\left|\frac{1}{b-a}\int_a^b f(x)dx-\frac{f\left(\frac{a+b}{2}\right)}{2}\right|\leq
\left|\frac{
f\left(\frac{3b-a}{2}\right)+
f\left(\frac{3a-b}{2}\right)}{4}\right|
\end{equation}
\end{lemma}
\begin{proof} By used the  change  of  the  variable   $x=\frac{3}{4}t+\frac{a+b}{4},\;t\in[\frac{3a-b}{3},\frac{3b-a}{3}]$ we get
\begin{equation}\label{PP}
\begin{split}
\frac{1}{b-a}\int_a^b f(x)dx&=\frac{3}{4(b-a)}\int_{\frac{3a-b}{3}}^{\frac{3b-a}{3}} f(\frac{3}{4}t+\frac{a+b}{4})dx\\
&=\frac{3}{4(b-a)}\int_{\frac{3a-b}{3}}^{\frac{3b-a}{3}}
f\left(\frac{1}{2}\left(\frac{3}{2}t\right)+\frac{1}{2}\left(\frac{a+b}{2}\right)\right)dx\\
&\leq\frac{3}{8(b-a)}\int_{\frac{3a-b}{3}}^{\frac{3b-a}{3}}
\left[f\left(\frac{3}{2}t\right)+f\left(\frac{a+b}{2}\right)\right]dx\\
&=\frac{f\left(\frac{a+b}{2}\right)}{2}+\frac{1}{4(b-a)}
\int_{\frac{3a-b}{2}}^{\frac{3b-a}{2}}f(t)dt.
\end{split}
\end{equation}
From the right hand side of inequality (\ref{1}), we have
\begin{equation}\label{QQ}
\frac{1}{4(b-a)}\int_{\frac{3a-b}{2}}^{\frac{3b-a}{2}}f(t)dt\leq
\frac{f\left(\frac{3a-b}{2}\right)+f\left(\frac{3b-a}{2}\right)}{4}.
\end{equation}
In view of (\ref{PP}) and (\ref{QQ}), we deduce that the inequalities (\ref{K1}) and (\ref{K2}) holds true.
\end{proof}
\section{Main results}
\begin{theorem}\label{TTTT}Let $f:I\subseteq\mathbb{R}\longrightarrow\mathbb{R}$
be a  differentiable  mapping  on $I^{0}$ with $a<b,$ and its derivative $f^\prime:[\frac{3a-b}{2},\frac{3b-a}{2}]\longrightarrow\mathbb{R},$ be a continuous on $[\frac{3a-b}{2},\frac{3b-a}{2}]$. Let $q\geq1,$ if $|f^\prime|^q$ is a convex function on $[\frac{3a-b}{2},\frac{3b-a}{2}]$, then the following inequality holds
\begin{equation}
\left|\frac{1}{b-a}\int_a^b f(x)dx-f\left(\frac{a+b}{2}\right)\right|\leq\frac{b-a}{8}\left(
\left|f^\prime\left(\frac{3a-b}{2}\right)\right|^q+
\left|f^\prime\left(\frac{3b-a}{2}\right)\right|^q\right)^{1/q}.
\end{equation}
\end{theorem}
\begin{proof}By means of Lemma \ref{l2}, we have
\begin{equation}\label{yy}
\frac{1}{2(b-a)}\int_{\frac{3a-b}{2}}^{\frac{3b-a}{2}} f(t)dt =f\left(\frac{a+b}{2}\right)+2(b-a)
\left[\int_0^{\frac{1}{2}} tf^\prime\left(\frac{3b-a}{2}+2(a-b)t\right)dt+
\int_{\frac{1}{2}}^1(t-1)f^\prime\left(\frac{3b-a}{2}+2(a-b)t\right)dt\right].
\end{equation}
From (\ref{PP}) and (\ref{yy}), we get
\begin{equation}\label{M}
\frac{1}{b-a}\int_a^bf(x)dx-f\left(\frac{a+b}{2}\right)\leq (b-a)\left[\int_0^{\frac{1}{2}} t\left|
f^\prime\left(\frac{3b-a}{2}+2(a-b)t\right)\right|dt+
\int_{\frac{1}{2}}^1(1-t)\left|f^\prime\left(\frac{3b-a}{2}+2(a-b)t\right)\right|dt\right].
\end{equation}
By the  power-mean  inequality, we find
\begin{equation}\label{MM}
\begin{split}
\int_0^{1/2}t\left|
f^\prime\left(\frac{3b-a}{2}+2(a-b)t\right)\right|dt&=\int_0^{1/2}t\left|f^\prime
\left(t(\frac{3a-b}{2})+(1-t)(\frac{3b-a}{2})\right)\right|dt\\
&\leq \left(\int_0^{1/2}t\;dt\right)^{1-1/q}\left(\int_0^{1/2}t\left|f^\prime
\left(t(\frac{3a-b}{2})+(1-t)(\frac{3b-a}{2})\right)\right|^q dt\right)^{1/q}\\
&\leq \left(\frac{1}{8}\right)^{1-1/q}\left(\int_0^{1/2}t^2\left|f^\prime(\frac{3a-b}{2})
\right|^q+t(1-t)\left|f(\frac{3b-a}{2})\right|^q\right)^{1/q}\\
&=\left(\frac{1}{8}\right)^{1-1/q}\left(
\frac{\left|f^\prime(\frac{3a-b}{2})\right|^q}{24}+\frac{\left|f(\frac{3b-a}{2})\right|^q}{12}\right)^{1/q}
\end{split}
\end{equation}
In the same way, we get
\begin{equation}\label{MMM}
\int_{\frac{1}{2}}^1(1-t)\left|f^\prime\left(\frac{3b-a}{2}+2(a-b)t\right)\right|dt\leq
\left(\frac{1}{8}\right)^{1-1/q}\left(\frac{\left|f^\prime(\frac{3a-b}{2})\right|^q}{12}+
\frac{\left|f(\frac{3b-a}{2})\right|^q}{24}\right)^{1/q}
\end{equation}
So, (\ref{M}), (\ref{MM}) and (\ref{MMM}) completes the proof of this Theorem.
\end{proof}
\begin{theorem}\label{1TTTT}Let $f:I\subseteq\mathbb{R}\longrightarrow\mathbb{R}$
be a  differentiable  mapping  on $I^{0}$ with $a<b,$ and its derivative $f^\prime:[\frac{3a-b}{2},\frac{3b-a}{2}]\longrightarrow\mathbb{R},$ be a continuous on $[\frac{3a-b}{2},\frac{3b-a}{2}]$. Let $q\geq1,$ if $|f^\prime|^q$ is a convex function on $[\frac{3a-b}{2},\frac{3b-a}{2}]$, then the following inequality holds
\begin{equation}
\left|\frac{1}{b-a}\int_a^b f(x)dx-f\left(\frac{a+b}{2}\right)\right|\leq(b-a)
\left(\frac{1}{2^{p+1}(p+1)}\right)^{\frac{1}{p}}\left(\frac{
\left|f^\prime\left(\frac{3a-b}{2}\right)\right|^q+
\left|f^\prime\left(\frac{3b-a}{2}\right)\right|^q}{2}\right)^{1/q}.
\end{equation}
\end{theorem}
\begin{proof} By again the H\"older's inequality we have
\begin{equation}\label{1267}
\begin{split}
\int_0^{1/2}t \left|f^\prime\left(t(\frac{3a-b}{2})+(1-t)(\frac{3b-a}{2})\right)\right|dt&\leq \left[\int_0^{1/2}t^pdt\right]^{\frac{1}{p}} \left[\int_0^{1/2}\left|f^\prime\left(t(\frac{3a-b}{2})+(1-t)(\frac{3b-a}{2})\right)\right|^q dt\right]^{\frac{1}{q}}\\
&\leq\left[\frac{1}{(p+1)2^{p+1}}\right]^{\frac{1}{p}}\left[\left|f^\prime(\frac{3a-b}{2})\right|^q\int_0^{1/2}t dt+\left|f^\prime(\frac{3b-a}{2})\right|^q\int_0^{1/2}(1-t) dt\right]^{\frac{1}{q}}\\
&=\left[\frac{1}{(p+1)2^{p+1}}\right]^{\frac{1}{p}}\left[\frac{\left|f^\prime(\frac{3a-b}{2})\right|^q+3\left|f^\prime(\frac{3b-a}{2})\right|^q}{8}\right]^{\frac{1}{q}}.
\end{split}
\end{equation}
Similarly, we get
\begin{equation}\label{1268}
\begin{split}
\int_{1/2}^1(1-t) \left|f^\prime\left(t(\frac{3a-b}{2})+(1-t)(\frac{3b-a}{2})\right)\right|dt&\leq \left[\int_0^{1/2}t^pdt\right]^{\frac{1}{p}} \left[\int_{1/2}^1\left|f^\prime\left(t(\frac{3a-b}{2})+(1-t)(\frac{3b-a}{2})\right)\right|^q dt\right]^{\frac{1}{q}}\\
&\leq\left[\frac{1}{(p+1)2^{p+1}}\right]^{\frac{1}{p}}\left[\left|f^\prime(\frac{3a-b}{2})\right|^q\int_{1/2}^1t dt+\left|f^\prime(\frac{3b-a}{2})\right|^q\int_{1/2}^1(1-t) dt\right]^{\frac{1}{q}}\\
&=\left[\frac{1}{(p+1)2^{p+1}}\right]^{\frac{1}{p}}\left[\frac{3\left|f^\prime(\frac{3a-b}{2})\right|^q+\left|f^\prime(\frac{3b-a}{2})\right|^q}{8}\right]^{\frac{1}{q}}.
\end{split}
\end{equation}
Thus, the inequalities (\ref{M}), (\ref{1267}) and (\ref{1268}) completes the proof of this Theorem.
\end{proof}
\begin{corollary}\label{c1}From  Theorem 2--3 we get the following inequality for $q > 1$ we get
$$\left|\frac{1}{b-a}\int_a^b f(x)dx-f\left(\frac{a+b}{2}\right)\right|\leq\min\{K_1,K_2\}(b-a)\left(
\left|f^\prime\left(\frac{3a-b}{2}\right)\right|^q+
\left|f^\prime\left(\frac{3b-a}{2}\right)\right|^q\right)^{1/q}$$
where $K_1=\frac{1}{8}$ and $K_2=\left(\frac{1}{(p+1)2^{p+1+
\frac{1}{pq}}}\right)^{\frac{1}{p}},$ such that $p=\frac{q}{q-1}.$
\end{corollary}
\begin{theorem}\label{TTTT1}Let $f:I\subseteq\mathbb{R}\longrightarrow\mathbb{R}$
be a  differentiable  mapping  on $I^{0}$ with $a<b,$ and its derivative $f^\prime:[\frac{3a-b}{2},\frac{3b-a}{2}]\longrightarrow\mathbb{R},$ be a continuous on $[\frac{3a-b}{2},\frac{3b-a}{2}]$. Let $q\geq1,$ if $|f^{\prime\prime}|^q$ is a convex function on $[\frac{3a-b}{2},\frac{3b-a}{2}]$, then the following inequality holds
\begin{equation}
\left|\frac{1}{b-a}\int_a^bf(x)dx-\frac{f\left(\frac{3a-b}{2}\right)+
f\left(\frac{3b-a}{2}\right)+2f\left(\frac{a+b}{2}\right)}{4}\right|\leq\frac{(b-a)^2}{3}\left[
\frac{\left|f^{\prime\prime}\left(\frac{3a-b}{2}\right)\right|^q+
\left|f^{\prime\prime}\left(\frac{3b-a}{2}\right)\right|^q}{2}\right]^{\frac{1}{q}}
\end{equation}
\end{theorem}
\begin{proof}From Lemma \ref{l3}, we have
\begin{equation}\label{rrr}
\frac{1}{2(b-a)}\int_{\frac{3a-b}{2}}^{\frac{3b-a}{2}} f(t)dt=\frac{f\left(\frac{3b-a}{2}\right)+f\left(\frac{3a-b}{2}\right)}{2}
-2(b-a)^2\int_0^1t(1-t) f^{\prime\prime}\left(t\left(\frac{3a-b}{2}\right)+(1-t)\left(
\frac{3b-a}{2}\right)\right)dt.
\end{equation}
Thus, by (\ref{PP}) and (\ref{rrr}) we obtain
\begin{equation}\label{maa}
\left|\frac{1}{b-a}\int_a^bf(x)dx-\frac{f\left(\frac{3a-b}{2}\right)+f\left(\frac{3b-a}{2}\right)+2f\left(\frac{a+b}{2}\right)}{4}\right|\leq2(b-a)^2\int_0^1t(1-t) \left|f^{\prime\prime}\left(t\left(\frac{3a-b}{2}\right)+(1-t)\left(
\frac{3b-a}{2}\right)\right)\right|dt.
\end{equation}
Suppose that $q>1$, from the H\"older's inequality for $q,\;p=\frac{q}{q-1}$ we get
$$\int_0^1(t-t^2) \left|f^{\prime\prime}\left(t\left(\frac{3a-b}{2}\right)+(1-t)\left(
\frac{3b-a}{2}\right)\right)\right|dt=$$
\begin{equation*}
\begin{split}
\;\;\;\;\;\;\;\;\;\;\;\;\;\;\;\;\;\;\;\;\;\;\;\;\;\;\;\;\;\;\;\;\;\;\;\;\;\;\;\;\;\;\;&=\int_0^1 [(t-t^2)]^{1-\frac{1}{q}}[(t-t^2)]^{\frac{1}{q}}
\left|f^{\prime\prime}\left(t\left(\frac{3a-b}{2}\right)+(1-t)\left(
\frac{3b-a}{2}\right)\right)\right|dt\\
&\leq\left[\int_0^1 (t-t^2)\right]^{\frac{q-1}{q}}\left[\int_0^1(t-t^2)\left|f^{\prime\prime}\left(t\left(\frac{3a-b}{2}\right)+(1-t)\left(
\frac{3b-a}{2}\right)\right)\right|^q dt\right]^{\frac{1}{q}}\\
&\leq \left(\frac{1}{6}\right)^{\frac{q-1}{q}}\left[\int_0^1(t^2-t^3)
\left|f^{\prime\prime}\left(\frac{3a-b}{2}\right)\right|^qdt+\int_0^1(t-t^2)(1-t)
\left|f^{\prime\prime}\left(
\frac{3b-a}{2}\right)\right|^q dt\right]^{\frac{1}{q}}\\
&\leq \left(\frac{1}{6}\right)^{\frac{q-1}{q}}
\left[\frac{\left|f^{\prime\prime}\left(\frac{3a-b}{2}\right)\right|^q
+\left|f^{\prime\prime}\left(\frac{3b-a}{2}\right)\right|^q}{12}\right]^{\frac{1}{q}}.
\end{split}
\end{equation*}
Now suppose that $q=1$, we obtain
\begin{equation*}
\begin{split}
\int_0^1(t-t^2) \left|f^{\prime\prime}\left(t\left(\frac{3a-b}{2}\right)+(1-t)\left(
\frac{3b-a}{2}\right)\right)\right|dt&\leq \int_0^1(t-t^2)
\left[t\left|f^{\prime\prime}\left(\frac{3a-b}{2}\right)\right|+(1-t)\left|f^{\prime\prime}
\left(\frac{3b-a}{2}\right)\right|\right] dt\\
&=\frac{\left|f^{\prime\prime}
\left(\frac{3a-b}{2}\right)\right|+\left|f^{\prime\prime}
\left(\frac{3b-a}{2}\right)\right|}{12},
\end{split}
\end{equation*}
which completes the proof.
\end{proof}
\begin{remark}  If $|f^{\prime\prime}(x)|\leq K$ on $[\frac{3a-b}{2},\frac{3b-a}{2}]$
in Theorem \ref{TTTT1}, we get
$$\left|\frac{1}{b-a}\int_a^bf(x)dx-\frac{f\left(\frac{3a-b}{2}\right)+
f\left(\frac{3b-a}{2}\right)+2f\left(\frac{a+b}{2}\right)}{4}\right|\leq\frac{K(b-a)^2}{3}.$$
\end{remark}

The other type is given by the next Theorem. For this, we note that the Beta function and gamma function are defined by

$$B(x,y)=\int_0^1 t^{x-1}(1-t)^{y-1}dt,\;x,y>0,\;\;\textrm{and}\;\;\Gamma(x)=
\int_0^\infty t^{x-1}e^{-t}dt,\;\;x>0.$$
The Beta function satisfied the following properties:
$$B(x,x)=2^{1-2x}B(1/2,x),\;\textrm{and}\;B(x,y)=\frac{\Gamma(x)\Gamma(y)}{\Gamma(x+y)}.$$
In particular, we have
$$B(p+1,p+1)=2^{1-2(p+1)}B(1/2,p+1)=2^{1-2(p+1)}\frac{\Gamma(1/2)\Gamma(p+1)}{\Gamma(p+3/2)}
=2^{1-2(p+1)}\frac{\sqrt{\pi}\Gamma(p+1)}{\Gamma(p+3/2)}.$$

\begin{theorem}\label{TTTTt1}Let $f:I^{0}\subseteq\mathbb{R}\longrightarrow\mathbb{R}$
be a  differentiable  mapping  on $I^{0}$ with $a<b,$ and its derivative $f^\prime:[\frac{3a-b}{2},\frac{3b-a}{2}]\longrightarrow\mathbb{R},$ be a continuous on $[\frac{3a-b}{2},\frac{3b-a}{2}]$. Let $q\geq1,$ if $|f^{\prime\prime}|^q$ is a convex function on $[\frac{3a-b}{2},\frac{3b-a}{2}]$, then the following inequality holds
\begin{equation}
\left|\frac{1}{b-a}\int_a^bf(x)dx-\frac{f\left(\frac{3a-b}{2}\right)+
f\left(\frac{3b-a}{2}\right)+2f\left(\frac{a+b}{2}\right)}{4}\right|\leq\frac{(b-a)^2}{2}
\left(\frac{\sqrt{\pi}\Gamma(p+1)}{2\Gamma(p+\frac{3}{2})}\right)^{\frac{1}{p}}\left(\frac{
\left|f^{\prime\prime}\left(\frac{3a-b}{2}\right)\right|^q+
\left|f^{\prime\prime}\left(\frac{3b-a}{2}\right)\right|^q}{2}\right)^{\frac{1}{q}},
\end{equation}
where $\frac{1}{p}+\frac{1}{q}=1.$
\end{theorem}
\begin{proof}Again, by the H\"older's inequality and using the fact that the function
$|f^{\prime\prime}|^q$ is convex we have
$$\int_0^1(t-t^2) \left|f^{\prime\prime}\left(t\left(\frac{3a-b}{2}\right)+(1-t)\left(
\frac{3b-a}{2}\right)\right)\right|dt$$
\begin{equation}\label{ma}
\begin{split}
&\leq\left[\int_0^1(t-t^2)^p\right]^{\frac{1}{p}}
\left[\int_0^1\left|f^{\prime\prime}\left(t\left(\frac{3a-b}{2}\right)+(1-t)\left(
\frac{3b-a}{2}\right)\right)\right|^qdt\right]^{\frac{1}{q}}\\
&\leq \left[\int_0^1(t-t^2)^p\right]^{\frac{1}{p}}
\left[\left|f^{\prime\prime}\left(\frac{3a-b}{2}\right)\right|^q\int_0^1tdt+
\left|f^{\prime\prime}\left(\frac{3b-a}{2}\right)\right|^q\int_0^1(1-t)dt\right]^{\frac{1}{q}}
\\
&=\left(\frac{\sqrt{\pi}\Gamma(p+1)}{2^{1+2p}\Gamma(p+\frac{3}{2}}\right)^{\frac{1}{p}}\left(\frac{
\left|f^{\prime\prime}\left(\frac{3a-b}{2}\right)\right|^q+
\left|f^{\prime\prime}\left(\frac{3b-a}{2}\right)\right|^q}{2}\right)^{\frac{1}{q}}.
\end{split}
\end{equation}
Finally, from (\ref{maa}) and (\ref{ma}) we obtain the desired result.
\end{proof}
\begin{remark} With the above assumptions given that $|f^{\prime\prime}|\leq K$ on
$[\frac{3a-b}{2},\frac{3b-a}{2}]$ we obtain
$$\left|\frac{1}{b-a}\int_a^bf(x)dx-\frac{f\left(\frac{3a-b}{2}\right)+
f\left(\frac{3b-a}{2}\right)+2f\left(\frac{a+b}{2}\right)}{4}\right|\leq\frac{K(b-a)^2}{2}
\left(\frac{\sqrt{\pi}\Gamma(p+1)}{2\Gamma(p+\frac{3}{2})}\right)^{\frac{1}{p}},
$$
where $\frac{1}{p}+\frac{1}{q}=1.$
\end{remark}

Another Hermite--Hadamard type inequality for powers in terms of the second derivatives is obtained as following:

\begin{theorem}\label{T56} With the assumptions of Theorem 5 we have the inequality:
\begin{equation}
\left|\frac{1}{b-a}\int_a^bf(x)dx-\frac{f\left(\frac{3a-b}{2}\right)+
f\left(\frac{3b-a}{2}\right)+2f\left(\frac{a+b}{2}\right)}{4}\right|\leq (b-a)^2\left(
\left|f^{\prime\prime}\left(\frac{3a-b}{2}\right)\right|^q+(q+1)
\left|f^{\prime\prime}\left(\frac{3b-a}{2}\right)\right|^q\right)^{\frac{1}{q}},
\end{equation}
where
$$K(p,q)=2\left(\frac{1}{p+1}\right)^{\frac{1}{p}}\left(\frac{1}{(q+1)(q+2)}\right)^{\frac{1}{q}}.$$
\end{theorem}
\begin{proof}From the H\"older's inequality we have
$$\int_0^1(t-t^2) \left|f^{\prime\prime}\left(t\left(\frac{3a-b}{2}\right)+(1-t)\left(
\frac{3b-a}{2}\right)\right)\right|dt$$
\begin{equation*}
\begin{split}
\;\;\;\;\;\;\;\;\;\;\;\;\;\;\;\;\;\;\;\;\;\;\;\;\;\;\;\;\;\;\;\;\;\;\;\;&\leq\left[\int_0^1t^pdt\right]^{\frac{1}{p}}
\left[\int_0^1(1-t)^q\left|f^{\prime\prime}\left(t\left(\frac{3a-b}{2}\right)+(1-t)\left(
\frac{3b-a}{2}\right)\right)\right|^qdt\right]^{\frac{1}{q}}\\
&\leq \left[\int_0^1t^pdt\right]^{\frac{1}{p}}
\left[\left|f^{\prime\prime}\left(\frac{3a-b}{2}\right)\right|^q\int_0^1t(1-t)^qdt+
\left|f^{\prime\prime}\left(\frac{3b-a}{2}\right)\right|^q\int_0^1(1-t)^{q+1}dt\right]^{\frac{1}{q}}\\
&=\left(\frac{1}{p+1}\right)^{\frac{1}{p}}\left[B(2,q+1)\left|f^{\prime\prime}\left(\frac{3a-b}{2}\right)\right|^q+\frac{
\left|f^{\prime\prime}\left(\frac{3b-a}{2}\right)\right|^q}{q+2}\right]^{\frac{1}{q}}\\
&=\left(\frac{1}{p+1}\right)^{\frac{1}{p}}\left(\frac{1}{(q+1)(q+2)}\right)^{\frac{1}{q}}\left(
\left|f^{\prime\prime}\left(\frac{3a-b}{2}\right)\right|^q+(q+1)
\left|f^{\prime\prime}\left(\frac{3b-a}{2}\right)\right|^q\right)^{\frac{1}{q}}.
\end{split}
\end{equation*}
So, the proof of Theorem \ref{T56} is completes.
\end{proof}

A similar result is embodied in the following Theorem.

\begin{theorem}\label{T57} et $f:I^{0}\subseteq\mathbb{R}\longrightarrow\mathbb{R}$
be a  differentiable  mapping  on $I^{0}$ with $a<b,$ and its derivative $f^\prime:[\frac{3a-b}{2},\frac{3b-a}{2}]\longrightarrow\mathbb{R},$ be a continuous on $[\frac{3a-b}{2},\frac{3b-a}{2}]$. Let $q\geq1,$ if $|f^{\prime\prime}|^q$ is a convex function on $[\frac{3a-b}{2},\frac{3b-a}{2}]$, then the following inequality holds:
\begin{equation}
\left|\frac{1}{b-a}\int_a^bf(x)dx-\frac{f\left(\frac{3a-b}{2}\right)+
f\left(\frac{3b-a}{2}\right)+2f\left(\frac{a+b}{2}\right)}{4}\right|\leq (b-a)^2\left(2
\left|f^{\prime\prime}\left(\frac{3a-b}{2}\right)\right|^q+(q+1)
\left|f^{\prime\prime}\left(\frac{3b-a}{2}\right)\right|^q\right)^{\frac{1}{q}},
\end{equation}
where
$$K_2(q)=\left(\frac{2}{(q+1)(q+2)(q+3)}\right)^{\frac{1}{q}}.$$
\end{theorem}
\begin{proof}From the  power--mean inequality we obtain
$$\int_0^1(t-t^2) \left|f^{\prime\prime}\left(t\left(\frac{3a-b}{2}\right)+(1-t)\left(
\frac{3b-a}{2}\right)\right)\right|dt$$
\begin{equation*}
\begin{split}
\;\;\;\;\;\;\;\;\;\;\;\;\;\;\;\;\;\;\;\;\;\;\;\;\;\;\;\;\;\;\;\;\;\;\;\;&\leq\left[\int_0^1tdt\right]^{1-\frac{1}{q}}
\left[\int_0^1t(1-t)^q\left|f^{\prime\prime}\left(t\left(\frac{3a-b}{2}\right)+(1-t)\left(
\frac{3b-a}{2}\right)\right)\right|^qdt\right]^{\frac{1}{q}}\\
&\leq \left[\int_0^1tdt\right]^{1-\frac{1}{q}}
\left[\left|f^{\prime\prime}\left(\frac{3a-b}{2}\right)\right|^q\int_0^1t^2(1-t)^qdt+
\left|f^{\prime\prime}\left(\frac{3b-a}{2}\right)\right|^q\int_0^1t(1-t)^{q+1}dt\right]^{\frac{1}{q}}\\
&=\left(\frac{1}{2}\right)^{1-\frac{1}{q}}\left[B(3,q+1)\left|f^{\prime\prime}\left(\frac{3a-b}{2}\right)\right|^q+B(2,q+2)
\left|f^{\prime\prime}\left(\frac{3b-a}{2}\right)\right|^q\right]^{\frac{1}{q}}\\
&=\left(\frac{1}{2}\right)^{1-\frac{1}{q}}\left(\frac{1}{(q+1)(q+2)(q+3)}\right)^{\frac{1}{q}}\left(2
\left|f^{\prime\prime}\left(\frac{3a-b}{2}\right)\right|^q+(q+1)
\left|f^{\prime\prime}\left(\frac{3b-a}{2}\right)\right|^q\right)^{\frac{1}{q}}.
\end{split}
\end{equation*}
which completes the proof of Theorem \ref{T57}.
\end{proof}
\begin{corollary} From  Theorem 5--7 we get the following inequality for $q > 1$
$$\left|\frac{1}{b-a}\int_a^bf(x)dx-\frac{f\left(\frac{3a-b}{2}\right)+
f\left(\frac{3b-a}{2}\right)+2f\left(\frac{a+b}{2}\right)}{4}\right|\leq\min(K_3, K_4, K_5, K_6) (b-a)^2,$$
where
\begin{equation*}
\begin{split}
K_3&=\frac{1}{3}\left(\frac{
\left|f^{\prime\prime}\left(\frac{3a-b}{2}\right)\right|^q+
\left|f^{\prime\prime}\left(\frac{3b-a}{2}\right)\right|^q}{2}\right)^{\frac{1}{q}}\\
K_4&=2
\left(\frac{\sqrt{\pi}\Gamma(p+1)}{2\Gamma(p+\frac{3}{2})}\right)^{\frac{1}{p}}\left(\frac{
\left|f^{\prime\prime}\left(\frac{3a-b}{2}\right)\right|^q+
\left|f^{\prime\prime}\left(\frac{3b-a}{2}\right)\right|^q}{2}\right)^{\frac{1}{q}}\\
K_5&=2\left(\frac{1}{p+1}\right)^{\frac{1}{p}}\left(\frac{1}{(q+1)(q+2)}\right)^{\frac{1}{q}}\left(
\left|f^{\prime\prime}\left(\frac{3a-b}{2}\right)\right|^q+(q+1)
\left|f^{\prime\prime}\left(\frac{3b-a}{2}\right)\right|^q\right)^{\frac{1}{q}}\\
K_6&=\left(\frac{2}{(q+1)(q+2)(q+3)}\right)^{\frac{1}{q}}\left(2
\left|f^{\prime\prime}\left(\frac{3a-b}{2}\right)\right|^q+(q+1)
\left|f^{\prime\prime}\left(\frac{3b-a}{2}\right)\right|^q\right)^{\frac{1}{q}}.
\end{split}
\end{equation*}
\end{corollary}
\section{Applications}
\subsection{Applications to special means}
Now using the results of Section 3, we give some applications to special means of positive real numbers.\\
\noindent 1. The arithmetic mean:
$$A=A(a,b)=\frac{a+b}{2};\;a,b\in\mathbb{R},\;\;\textrm{with}\;\;a,b>0.$$
\noindent 2. The geometric mean:
$$G=G(a,b)=\sqrt{ab};\;a,b\in\mathbb{R},\;\;\textrm{with}\;\;a,b>0.$$
\noindent 3. The logarithmic  mean:
$$L(a,b)=\frac{b-a}{\log b-\log a}.$$
\noindent 4.  The generalized logarithmic mean:
$$L_n(a,b)=\left[\frac{b^{n+1}-a^{n+1}}{(b-a)(n+1)}\right]^{\frac{1}{n}};\;n
\in\mathbb{Z}\setminus \{-1,0\},\;a,b\in\mathbb{R},\;\;\textrm{with}\;\;a,b>0.$$

\begin{propo}Let $n\in\mathbb{Z}\setminus\{-1,0\}$ and $a,b\in\mathbb{R}$ such that $0<a<b.$
Then the hollowing inequalities
\begin{equation}
\left|2L_n^n(a,b)-A^n(a,b)\right|\leq A\left(\left|\frac{3b-a}{2}\right|^n,
\left|\frac{3a-b}{2}\right|^n\right),
\end{equation}
and
\begin{equation}
\left|A^n(a,b)-L_n^n(a,b)\right|\leq\min\{K_1,K_2\}(2^{\frac{1}{q}}|n|(b-a))\left[A\left(\left|\frac{3a-b}{2}\right|^{(n-1)q},
\left|\frac{3b-a}{2}\right|^{(n-1)q}\right)\right]^{1/q},
\end{equation}
holds.
\end{propo}
\begin{proof}
The assertion follows from Theorem \ref{TT1} and Theorem \ref{TTTT}--\ref{1TTTT} for $f(x)=x^n$ and $n$ as specified above.
\end{proof}
\begin{propo}Let  $a,b\in\mathbb{R}$ such that $0<a<b.$ Then the following inequalities
\begin{equation}
\left|2G^{-2}(a,b)-A^{-2}(a,b)\right|\leq
A\left(\left(\frac{3a-b}{2}\right)^{-2},\left(\frac{3b-a}{2}\right)^{-2}\right),
\end{equation}
and
\begin{equation}
\left|G^{-2}(a,b)-A^{-2}(a,b)\right|\leq\min\{K_1,K_2\}(4^{\frac{1}{q}}(b-a))
\left[A\left(\left|\frac{3a-b}{2}\right|^{-3q},\left|\frac{3b-a}{2}\right|^{-3q}\right)\right]^{1/q},
\end{equation}
are valid.
\end{propo}
\begin{proof}
 The  assertion  follows from  Theorem \ref{TT1} and Theorem \ref{TTTT}--\ref{1TTTT} applied  for $f(x)=\frac{1}{x^2}.$
\end{proof}
\begin{propo}Let $q\geq1$ and $a,b\in\mathbb{R}$ such that $0<a<b.$ Then the following inequalities
\begin{equation}
\left|A^{-1}(a,b)-2L^{-1}(a,b)\right|\leq A\left(\left|\frac{3a-b}{2}\right|^{-1},
\left|\frac{3b-a}{2}\right|^{-1}\right),
\end{equation}
and
\begin{equation}
\left|A^{-1}(a,b)-L^{-1}(a,b)\right|\leq\min\{K_1,K_2\}(2^{\frac{1}{q}}(b-a))\left
[A\left(\left|\frac{3a-b}{2}\right|^{-2q},\left|\frac{3b-a}{2}\right|^{-2q}\right)\right]^{1/q},
\end{equation}
holds true.
\end{propo}
\begin{proof}
The assertion follows from Theorem \ref{TT1} and Theorem \ref{TTTT}--\ref{1TTTT} for $f(x)=\frac{1}{x}.$
\end{proof}
\subsection{The midpoint formula}
Let $d$ be a partition $a=x_0<x_1<...<x_{m-1}<x_m=b$ of the interval $[a,b]$ and consider the quadrature  formula
$$\int_a^b f(x)dx=T_i(f,d)+E_i(f,d),\;i=1,2,$$
where
$$T_1(f,d)=\sum_{i=0}^{m-1}\frac{f(x_i)+f(x_{i+1})}{2}(x_{i+1}-x_{i})$$
for  the  trapezoidal  version  and
$$T_2(f,d)=\sum_{i=0}^{m-1}f\left(\frac{x_i+x_{i+1}}{2}\right)(x_{i+1}-x_{i})$$
for  the  midpoint  version  and  $E_i(f, d)$  denotes  the  associated  approximation  error.
\begin{propo}Suppose that the function $f$ is convex, then for every partition of $[a,b]$ the  midpoint  error  satisfies
\begin{equation}
\begin{split}
\left|\int_a^b f(x)dx+E_2(f,d)\right|&\leq\sum_{i=0}^{m-1}(x_{i+1}-x_i)\frac{\left|f\left(\frac{3x_i-x_{i+1}}{2}\right)+f\left(\frac{3x_{i+1}-x_i}{2}\right)\right|}{2}\\
&\leq\sum_{i=1}^{m-1}(x_{i+1}-x_i)\max\left(\left|f\left(\frac{3x_i-x_{i+1}}{2}\right)\right|,\left|f\left(\frac{3x_{i+1}-x_i}{2}\right)\right|\right).
\end{split}
\end{equation}
\end{propo}
\begin{proof}From Theorem \ref{TT1}, we have
$$\left|2\int_{x_i}^{x_{i+1}} f(x)dx-(x_{i+1}-x_i)f\left(\frac{x_i+x_{i+1}}{2}\right)\right|\leq \left|\frac{f\left(\frac{3x_i-x_{i+1}}{2}\right)+f\left(\frac{3x_{i+1}-x_i}{2}\right)}{2}\right|.$$
On the other hand, we have
\begin{equation}
\begin{split}
\left|\int_a^b f(x)dx+\left\{\int_a^b f(x)dx-T_2(f,d)\right\}\right|&=\left|\sum_{i=0}^{m-1}\left\{2\int_{x_i}^{x_{i+1}} f(x)dx-(x_{i+1}-x_i)f\left(\frac{x_i+x_{i+1}}{2}\right)\right\}\right|\\
&\leq\sum_{i=0}^{m-1}(x_{i+1}-x_i)\frac{\left|f\left(\frac{3x_i-x_{i+1}}{2}\right)+f\left(\frac{3x_{i+1}-x_i}{2}\right)\right|}{2}\\
&\leq\sum_{i=1}^{m-1}(x_{i+1}-x_i)\max\left(\left|f\left(\frac{3x_i-x_{i+1}}{2}\right)\right|,\left|f\left(\frac{3x_{i+1}-x_i}{2}\right)\right|\right).
\end{split}
\end{equation}
\end{proof}
\begin{propo}Suppose that the function $|f^\prime|^q,\;q\leq1$, then for every partition of $[a,b]$ the  midpoint  error  satisfies
\begin{equation}
\begin{split}
\left|E_2(f;d)\right|&\leq\min(K_1,K_2)\sum_{i=1}^{m-1}(x_{i+1}-x_i)^2\left[\left|f^\prime\left(\frac{3x_i-x_{i+1}}{2}\right)\right|^q+\left|f^\prime\left(\frac{3x_{i+1}-x_i}{2}\right)\right|^q\right]^{\frac{1}{q}}\\
&\leq2\min(K_1,K_2)\sum_{i=1}^{m-1}(x_{i+1}-x_i)^2\max\left(\left|f^\prime\left(\frac{3x_i-x_{i+1}}{2}\right)\right|,\left|f^\prime\left(\frac{3x_{i+1}-x_i}{2}\right)\right|\right).
\end{split}
\end{equation}
\end{propo}
\begin{proof}From Corollary \ref{c1}, we obtain
$$\left|\int_{x_i}^{x_{i+1}} f(x)dx-(x_{i+1}-x_i)f\left(\frac{x_i+x_{i+1}}{2}\right)\right|\leq\min(K_1,K_2)(x_{i+1}-x_i)^2\left[\left|f^\prime\left(\frac{3x_i-x_{i+1}}{2}\right)\right|^q+\left|f^\prime\left(\frac{3x_{i+1}-x_i}{2}\right)\right|^q\right]^{\frac{1}{q}}.$$
On the other hand, we have
\begin{equation}
\begin{split}
\left|\int_a^b f(x)dx-T_2(f,d)\right|&=\left|\sum_{i=0}^{m-1}\left\{\int_{x_i}^{x_{i+1}} f(x)dx-(x_{i+1}-x_i)f\left(\frac{x_i+x_{i+1}}{2}\right)\right\}\right|\\
&\leq\min(K_1,K_2)\sum_{i=0}^{m-1}(x_{i+1}-x_i)^2\left[\left|f^\prime\left(\frac{3x_i-x_{i+1}}{2}\right)\right|^q+\left|f^\prime\left(\frac{3x_{i+1}-x_i}{2}\right)\right|^q\right]^{\frac{1}{q}}\\
&\leq2\min(K_1,K_2)\sum_{i=1}^{m-1}(x_{i+1}-x_i)^2\max\left(\left|f^\prime\left(\frac{3x_i-x_{i+1}}{2}\right)\right|,\left|f^\prime\left(\frac{3x_{i+1}-x_i}{2}\right)\right|\right).
\end{split}
\end{equation}
\end{proof}
\subsection{Inequalities for some special functions}

\subsubsection{Modified Bessel functions}

Recall that the modified Bessel function $I_p$ of the first kind has the series representation [\cite{wat}, p. 77]
$$I_p(x)=\sum_{n\geq0}\frac{(x/2)^{p+2n}}{n!\Gamma(p+n+1)},$$
where $x\in\mathbb{R},$ while the modified Bessel function of the second kind $K_p$ [\cite{wat}, p. 78]
is usually defined as
$$K_p(x)=\frac{\pi}{2}\frac{I_{-p}(x)+I_p(x)}{\sin p\pi}.$$
For this we consider the function $\mathcal{I}_{p}:\mathbb{R}\longrightarrow[1,\infty),$  defined by
\begin{equation*}
\mathcal{I}_{p}(x)=2^p\Gamma(p+1)x^{-\nu}I_{p}(x),
\end{equation*}
where $\Gamma(.)$ is the gamma function.
\begin{propo}Let $p>-1,\;a,b\in\mathbb{R}$ such that $0<a<b$, then the following inequality holds
\begin{equation}\label{I1}
\left|\frac{\mathcal{I}_p(b)-\mathcal{I}_p(b)}{b-a}\right|\leq\frac{\left[(\frac{3a-b}{2})\mathcal{I}_{p+1}(\frac{3a-b}{2})+(\frac{3b-a}{2})\mathcal{I}_{p+1}(\frac{3b-a}{2})\right]+(a+b)\mathcal{I}_{p+1}(\frac{a+b}{2})}{8(p+1)}.
\end{equation}
In particular, the following inequality
\begin{equation}\label{I11}
\left|\frac{\cosh(b)-\cosh(a)}{b-a}\right|\leq\frac{\sinh\left(\frac{3a-b}{2}\right)+\sinh\left(\frac{3b-a}{2}\right)+2\sinh\left(\frac{a+b}{2}\right)}{4}.
\end{equation}
is true.
\end{propo}
\begin{proof}Observe that the function $x\mapsto\mathcal{I}_p^\prime(x)$ is convex on $[0,\infty),$ since as power series has only positive coefficients. Now, from Theorem \ref{TT1} we obtain
\begin{equation}
\left|\frac{\mathcal{I}_p(b)-\mathcal{I}_p(b)}{b-a}\right|\leq\frac{2\mathcal{I}_p^\prime\left(\frac{a+b}{2}\right)+\mathcal{I}_p^\prime\left(\frac{3a-b}{2}\right)+\mathcal{I}_p^\prime\left(\frac{3b-a}{2}\right)}{4}.
\end{equation}
By using the differentiation formula [\cite{wat}, p. 79]
\begin{equation}\label{mm}
\mathcal{I}_{p}^{\prime}(x)=\frac{x}{2(p+1)}\mathcal{I}_{p+1}(x)
\end{equation}
we deduce that the inequality (\ref{I1}) holds. Now taking into account the relations $\mathcal{I}_{-1/2}(x)=\cosh(x)$ and $\mathcal{I}_{1/2}(x)=\sinh(x)/x,$ the inequality (\ref{I1})  reduce to inequality (\ref{I11}).
\end{proof}
\begin{propo}Let $a,b\in\mathbb{R}$ such that $0<a<b.$ Suppose that $3a\neq b,$ then the following inequality holds true
\begin{equation}\label{II}
\left|\frac{a^pK_p(b)-b^pK(a)}{(ab)^p(b-a)}\right|\leq \frac{F_{p}(a,b)}{[(a+b)(3a-b)(3b-a)]^p},
\end{equation}
for all $p>1$, where
$$F_p(a,b)=2^{p+1}[(3a-b)(3b-a)]^pK_{p+1}(\frac{a+b}{2})+[2(a+b)(3b-a)]^pK_{p+1}(\frac{3a-b}{2})+[2(a+b)(3a-b)]^pK_{p+1}(\frac{3b-a}{2}).$$
\end{propo}
\begin{proof}Using the integral representation \rm{\cite{wat}, p. 181}
\[
K_{p}(x)=\int_{0}^{\infty}e^{-x\cosh t}\cosh(pt)dt,\, x>0\]
where $p\in\mathbb{R},$we conclude that the function $x\longmapsto K_{p}(x)$
is scompletely monotonic on $(0,\infty)$ for each $p\in\mathbb{R}.$ By the Leibniz formula for derivatives the product of two completely monotonic functions
is completely monotonic, we conclude that the function $x\mapsto\frac{K_{p}(x)}{x^{p}}$
is strictly completely monotonic on $(0,\infty)$ for all $p>1.$ Now, let $f_p(x)=-(\frac{K_{p}(x)}{x^{p}})^\prime,$ so the function $f_p(x)$ is completely monotonic on $(0,\infty)$ for all $p>1,$ and consequently is convex. Using Theorem \ref{TT1} and the fact that [\cite{wat}, p. 79] $(\frac{K_{p}(x)}{x^{p}})^\prime=-\frac{K_{p+1}(x)}{x^{p}}$, we conclude that the inequality (\ref{II}) is holds for all $p>1.$
\end{proof}
\subsection{$q-$digamma function}
Let $0<q<1$, the $q-$digamma function $\psi_q$, is the  $q-$analogue  of  the  psi  or  digamma  function $\psi$ defined by
\begin{equation}\label{int}
\begin{split}
\psi_q(x)
&=-\ln(1-q)+\ln q \sum_{k=0}^{\infty}\frac{q^{k+x}}{1-q^{k+x}}\\
&=-\ln(1-q)+\ln q \sum_{k=1}^{\infty}\frac{q^{kx}}{1-q^{k}}.
\end{split}
\end{equation}

For $q>1$ and $x>0$, the $q-$digamma function $\psi_q$ is defined by
\begin{equation*}
\begin{split}
\psi_q(x)&=-\ln(q-1)+\ln q\left[x-\frac{1}{2}-\sum_{k=0}^{\infty}\frac{q^{-(k+x)}}{1-q^{-(k+x)}}\right]\\
&=-\ln(q-1)+\ln q\left[x-\frac{1}{2}-\sum_{k=1}^{\infty}\frac{q^{-kx}}{1-q^{-kx}}\right].
\end{split}
\end{equation*}

\begin{propo}For $a,b\in \mathbb{R}$, such that $0<a<b$. Then the following inequality
\begin{equation}
\left|\frac{\psi_q(b)-\psi_q(b)}{b-a}\right|\leq \frac{\left[\psi_q^\prime\left(\left|\frac{3a-b}{2}\right)\right|+\psi_q^\prime\left(\frac{3b-a}{2}\right)+2\psi_q^\prime\left(\frac{a+b}{2}\right)\right]}{4}
\end{equation}
holds true for all $q>0.$
\end{propo}
\begin{proof} By using the definitions of the function $\psi_q(x)$ we deduce that the function $x\mapsto \psi_q^\prime(x)$ is completely monotonic on $(0,\infty)$ for all $q>0$, and consequently the function $x\mapsto \psi_q^\prime(x)$ is convex on $(0,\infty)$. Now applying Theorem \ref{TT1} we deduce that the inequality is valid for all $q>0.$
\end{proof}

\begin{propo}For $a,b\in \mathbb{R}$, such that $0<a<b$. Then the following inequality
\begin{equation}
\left|\frac{\psi_q(b)-\psi_q(b)}{b-a}-\frac{\left[\psi_q^\prime\left(\left|\frac{3a-b}{2}\right)\right|+\psi_q^\prime\left(\frac{3b-a}{2}\right)+2\psi_q^\prime\left(\frac{a+b}{2}\right)\right]}{4}\right|\leq
\frac{(b-a)^2\left[\psi_q^{(3)}\left(\left|\frac{3a-b}{2}\right|\right)+\psi_q^{(3)}\left(\frac{3b-a}{2}\right)\right]}{6},
\end{equation}
holds true for all $q>0.$
\end{propo}
\begin{proof}We set $f=\psi_q^\prime,$ thus the function $f^{\prime\prime}=\psi_q^{(3)}$ is completely monotonic on $(0,\infty)$ for all $q>0.$ Apply Theorem 4 we obtain the desired inequality.
\end{proof}
\vskip 3mm
\textbf{{Concluding Remarks:}}
Lastly, we conclude this paper by remarking that, we have
obtained new type integral
inequalities for convex functions and some of their applications.
All the inequalities are interesting and important in the field of integral inequalities. 

%\vskip 10mm
%\noindent{\bf Conflict of Interests:}
%\noindent The authors  declare  that there is no conflict of interests.\vskip 2mm
%\noindent{\bf Authors contributions:}\vskip 2mm

%\noindent All authors contributed equally to the manuscript. All authors read and approved the final manuscript.\\
%\textbf{Acknowledgements:}\\
%The authors would like to express profound gratitude to referees for deeper
%review of this paper and the referee's useful suggestions that led to an improved presentation of the paper.\\

\end{document}